\documentclass[12pt]{article}
\usepackage[utf8]{inputenc}

\usepackage{amssymb,amsmath,amsthm,epsfig,tabularx}
\usepackage{subfigure}
\usepackage{float}
\usepackage{url,epsf,graphicx,tikz}
\usetikzlibrary{decorations.markings}

\newtheorem{theorem}{Theorem}
\newtheorem{proposition}[theorem]{Proposition}

\newtheorem{question}[theorem]{Question}

\newcommand{\g}{{\rm \gamma}}

\begin{document}

\title{Domination versus independent domination in regular graphs}

\author
{ Martin Knor\thanks{Slovak University of Technology in Bratislava,
Faculty of Civil Engineering, Department of Mathematics, Bratislava,
Slovakia. E-Mail: \texttt{knor@math.sk}},
\quad Riste
\v{S}krekovski\thanks{FMF, University of Ljubljana \& Faculty of Information Studies, Novo mesto \& FAMNIT, University of Primorska, Koper, Slovenia.  
E-Mail: \texttt{skrekovski@gmail.com}},\quad Aleksandra Tepeh\thanks{Faculty
of Information Studies, Novo mesto \& Faculty of Electrical
Engineering and Computer Science, University of Maribor, Slovenia.
E-Mail: \texttt{aleksandra.tepeh@gmail.com}} }

\maketitle

\begin{abstract}

A set $S$ of vertices in a graph $G$ is a dominating set if every vertex of $G$ is in $S$ or is adjacent to a vertex in $S$. If, in addition, $S$ is an independent set, then $S$ is an independent dominating
set. The domination number $\g(G)$ of $G$ is the minimum cardinality of a dominating set
in $G$, while the independent domination number $i(G)$ of $G$ is the minimum cardinality
of an independent dominating set in $G$. We prove that for all integers $k \geq 3$ it holds that if $G$ is a connected $k$-regular graph, then $\frac{i(G)}{\g(G)} \leq \frac{k}{2}$, with equality if and only if $G = K_{k,k}$. The result was previously known only for $k\leq 6$. This affirmatively answers a question of Babikir and Henning from \cite{babi}.

\end{abstract}


\tikzset{My Style/.style={draw, circle, fill=black,scale=0.3}} 


\section{Introduction}

Given a graph $G = (V,E)$ an
\textit{independent} set is a subset of vertices $U \subseteq V$, such that no two vertices in $U$ are adjacent. An independent set is \textit{maximal} if no vertex can be added without violating independence. An independent set of maximum cardinality is called a \textsl{maximum independent} set.
A set $S$ of vertices in a graph $G$ is a \textit{dominating} set if every
vertex of $G$ is in $S$ or is adjacent to a vertex in $S$.
If, in addition, $S$ is an independent set, then $S$ is an independent dominating set.
The \textit{domination number} of $G$, denoted $\g(G)$, is the minimum cardinality
of a dominating set of $G$.
The \textit{independent domination number} of $G$, denoted $i(G)$, is the minimum cardinality of an
independent dominating set in $G$.
Note that an independent set of vertices in a graph $G$ is a dominating set of $G$ if and only if it is
a maximal independent set.
Therefore, $i(G)$ is equal to the minimum cardinality of a maximal
independent set of vertices in $G$.
The object of study in this paper are $k$-regular graphs,
i.e. such that every vertex has degree $k$. 

Dominating and independent dominating sets have been studied extensively in the literature; see for example the books \cite{Hay1,Hay2} and a survey \cite{GHsur}. In early studies authors considered the difference between $\g(G)$ and $i(G)$, \cite{bare,cock-hed,GHLS,kost}. 
In  \cite{GHLS} the authors initiated the study of the ratio $\frac{i(G)}{\g(G)}$. They showed that this ratio is at most $\frac{3}{2}$ for connected cubic graphs $G$, with
equality if and only if $G = K_{3,3}$. Southey and Henning \cite{SH} proved that the $\frac{3}{2}$-ratio can
be strengthened to a $\frac{4}{3}$-ratio if $K_{3,3}$ is excluded. In \cite{O} O and West constructed an infinite family of connected cubic graphs $G$ such that $\frac{i(G)}{\g (G)} = \frac{5}{4}$. A question of determining whether $\frac{4}{3}$-ratio from the above mentioned result of Southey and Henning can be improved to a $\frac{5}{4}$-ratio if finitely many graphs are forbidden, remains open.

The ratio of the independent domination number to the domination number for general graphs was studied by Furuya et al.~\cite{furuya} who showed that for a graph $G$ this ratio is at most $\Delta(G)-2\sqrt{\Delta(G)}+2$, where $\Delta(G)$ denotes the maximum degree of $G$.

\medskip

In this paper we give the affirmative answer to the following question from \cite{babi}, with which the above sharp bound of Furuya et al.~is also improved in the case of connected $k$-regular graphs.

\begin{question}
\label{Q}
Is it true that for all integers $k \geq 3$ if $G$ is a connected $k$-regular graph, then $\frac{i(G)}{\g(G)} \leq \frac{k}{2}$, with equality if and only if $G = K_{k,k}$?
\end{question}

As mentioned above, Question~{\ref{Q}} was already answered affirmatively for
$k=3$ in \cite{GHLS}, and for $k\in \{4,5,6\}$ in \cite{babi}.


\section{Proof of the main theorem}

In the proof of our theorem we use $G[A]$ to denote the subgraph of $G$ induced by a vertex set $A\subseteq V$, and we also use
an old result by Rosenberg \cite{rose64}.

\begin{proposition}\label{rose}
If $G$ is a regular graph of order $n$ with no isolated vertex, then 
$i(G) \leq \frac{n}{2}$.
\end{proposition}

\begin{theorem} For $k \geq 3$, if $G$ is a connected $k$-regular graph, 
then $$\frac{i(G)}{\g(G)} \leq \frac{k}{2},$$ 
with equality if and only if $G = K_{k,k}$.
\end{theorem}

\begin{proof} 
Let $G=(V,E)$ be a connected $k$-regular graph, $k \geq 3$. Let $A$ be a dominating set in $G$ with $|A|=\g(G)$, and $B=V\setminus A$. 
We distinguish two cases with respect to the number of edges in $G[A]$, which we denote by $s$.

\medskip

{\bf Case 1:} $s \geq \frac{\g(G)}{2}$. Denote by $e$ the number of edges having one end-vertex in $A$ and the other in $B$. Since $G$ is $k$-regular, we derive
$$
e=k\cdot |A|-2s=k\g(G)-2s \leq k\g(G)-\g(G)=(k-1)\g(G).
$$
Since $A$ is a dominating set this readily implies that $|B|\leq e \leq (k-1)\g(G)$.
We now estimate $n=|A|+|B|\leq \g(G)+(k-1)\g(G)=k\g(G)$, and using Proposition~\ref{rose},
we derive
$$
\frac{i(G)}{\g(G)} \leq \frac{n/2}{n/k} =\frac{k}{2}.
$$

\medskip

{\bf Case 2:} $s<\frac{\g(G)}{2}$. Let $A'$ denote a maximum independent set in $G[A]$. Clearly, $|A'|\geq |A|-s=\g(G)-s$. Let $|A'|= \g(G)-s+x$ for some $x\geq 0$. Then $A\setminus A'$ contains $s-x$ vertices which we denote by $b_1,b_2,\ldots, b_{s-x}$. Let $B_i$ be the set of neighbors of $b_i$ in $B$, $i\in \{1,2,\ldots,s-x\}$, and let 
$B'= B_1\cup B_2\cup \cdots\cup B_{s-x}$.
Then $|B'|\leq (s-x)(k-1)$, and since $k\geq 3$ we derive
$$\begin{array}{rcl}
|A'\cup B'| & = & |A'|+|B'|\\
            & \leq & \g(G)-s+x + (s-x)(k-1)\\
            & = & \g(G)-2s+sk+x(2-k)\\
						& \leq & \g(G)+(k-2)s.
\end{array}$$

\noindent Our next aim is to show that $A'\cup B'$ contains an independent dominating set of $G$. Let $B''$ be the set of vertices in $B'$ that have a neighbor in $A'$, and $C$ an independent dominating set in the subgraph of $G$ induced by the set $B'\setminus B''$. It is straightforward to verify that $I=A'\cup C$ is an independent dominating set in $G$.
Therefore we obtain
$$
i(G)\leq |I|= |A'|+|C| \leq |A'|+|B'| \leq  \g(G)+(k-2)s.
$$
Recall that $s<\frac{\g(G)}{2}$.
Now we consider the following cases with respect to the parity of $\g(G)$.
If $\g(G)$ is even, then $s\leq \frac{\g(G)}{2}-1$, and we have
$$
i(G)\leq \g(G)+(k-2)\left(\frac{\g(G)}{2}-1\right)=\frac{k}{2}\g(G)+2-k< \frac{k\g(G)}{2},
$$
and therefore 
$$
\frac{i(G)}{\g(G)} < \frac{\frac{k\g(G)}{2}
}{\g(G)} =\frac{k}{2}.
$$

\noindent If $\g(G)$ is odd, then $s\leq \frac{\g(G)-1}{2}$, and we obtain
$$
i(G)\leq \g(G)+(k-2)\frac{\g(G)-1}{2}=\frac{k\g(G)}{2}-\frac{k-2}{2}< \frac{k\g(G)}{2},
$$
which again implies the desired inequality.

\bigskip

Now we describe the extremal graphs, i.e.~graphs $G$ with $\frac{i(G)}{\gamma(G)}=\frac k2$.
Since $\frac{i(G)}{\gamma(G)}<\frac k2$ if there are less than $\gamma(G)/2$
edges in $G[A]$ (see calculations in Case $2$) the extremal graphs
can be obtained only if there are at least $\gamma(G)/2$ edges in $G[A]$.
In fact, $G[A]$ must have exactly $\gamma(G)/2$ edges, i.e.~$s=\gamma(G)/2$,
since otherwise we get $e<(k-1)\gamma(G)$ which implies
$n<k\gamma(G)$ and consequently $\frac{i(G)}{\gamma(G)}<\frac k2$.

Furthermore, we will show that $G[A]$ is a collection of independent edges.
Suppose that $G[A]$ has exactly $t$ components.
Take one vertex from each component to form an independent set.
This set can be completed with at most $(k-1)(\gamma(G)-t)$ vertices of $B$
(that is, with at most $k-1$ vertices of $B$ for each of non-selected
vertices from $A$) to a maximal independent set $M$.
Recall that $M$ is an independent dominating set of $G$ and therefore $i(G)\leq |M|$.
If $t>\gamma(G)/2$, then 
$$
\begin{array}{ccccl}
i(G) & \leq & |M| & \leq & t+(k-1)(\g(G)-t)\\
&&                & =    & k\g(G)-\g(G)-t(k-2)\\
&&                & <    & k\g(G)-\g(G)-\frac{\g(G)}{2}(k-2)\\
&&						    & =    & \frac k2 \g(G),
\end{array}
$$
which implies $\frac{i(G)}{\gamma(G)}<\frac k2$, a contradiction.
Therefore $t\leq \gamma(G)/2$, i.e.~$t\leq s$. 
If a component of $G[A]$ contains a cycle, then $s>\g(G)-t$,
which together with $s=\g(G)/2$ implies that $t>\gamma(G)/2$, a contradiction.
Thus $G[A]$ is a forest.
Since $t\leq s$ we have $s=\g(G)-t\geq \g(G)-s=s$, and thus $t=s$.

If there is a component of $G[A]$ which contains at least $3$ vertices, then
this component contains two independent vertices (recall that $G[A]$ is a
forest).
So take two independent vertices from this component and one vertex from
every other component of $G[A]$.
This set contains $t+1$ independent vertices and analogously as above it can
be completed with at most $(k-1)(\g(G)-t-1)$ vertices of $B$ to a maximal
independent set $M$.
We get
$$
\begin{array}{ccccl}
i(G) & \leq & |M| & \leq & (t+1)+(k-1)(\g(G)-t-1)\\
&&                & =    & k\g(G)-\g(G)-(t+1)(k-2)\\
&&                & =    & \frac k2\g(G)-(k-2)\\
&&		  & <    & \frac k2 \g(G),
\end{array}
$$
which implies $\frac{i(G)}{\gamma(G)}<\frac k2$, a contradiction.
Thus, $G[A]$ is a collection of independent edges $\{u_1v_1,u_2v_2,\ldots ,u_sv_s\}$.

Our next aim is to show that for each $i\in \{1,2,\ldots ,s\}$ vertices $u_i$ and $v_i$ do not have a common neighbor.
If there is $z\in N(u_i)\cap N(v_i)$, then taking $u_i$ and one vertex from every other edge of $G[A]$ to an independent set, we can complete it to a maximal independent set $I$ with at most $k-2$ neighbors of $v_i$ in $V\setminus A$, and at most $k-1$ neighbors for each of non-selected vertices of $G[A]$. Therefore
$$i(G)\leq |I| \leq \frac{\g(G)}{2}+k-2+\left(\frac{\g(G)}{2}-1\right)(k-1)=\frac{\g(G)}{2}k-1<\frac{\g(G)}{2}k,$$

\noindent which again yields $\frac{i(G)}{\gamma(G)}<\frac k2$.
Thus, $N(u_i)\cap N(v_i)=\emptyset$.

Now suppose that there is $z\in N(u_1)$ which has a neighbor $w$ outside $N(u_1)\cup N(v_1)$.
We distinguish two cases.

\medskip

{\bf Case A:} $w\in A\setminus \{u_1,v_1\}$.
Without loss of generality we may assume that $w=v_2$.
Then put to an independent set $v_1$ and $v_2$, and complete it to an independent dominating set $I$ analogously as above. More precisely, $I$ contains $v_i$ for every $i\in \{1,2,\ldots,s\}$, at most $k-1$ neighbors of $u_i$ in $B$ for every $i\in \{2,3,\ldots,s\}$,
and at most $k-2$ neighbors of $u_1$ in $B$ since $v_2z\in E$ and $z\in N(u_1)$. Then
$$i(G)\leq  |I| \leq  \frac{\g(G)}{2}+k-2+\left(\frac{\g(G)}{2}-1 \right)(k-1)<\frac{\g(G)}{2}k,$$
which again implies $\frac{i(G)}{\gamma(G)}<\frac k2$, a contradiction.

\medskip

{\bf Case B:}
$w\in V\setminus (A\cup N(u_1)\cup N(v_1))$.
Without loss of generality we may assume that $w$ is a neighbor of $u_2$.
Let $I$ be the set consisting of $v_i$ for every $i\in \{1,2,\ldots,s\}$,
at most $k-1$ neighbors of $u_i$ in $B$ for every $i\in \{3,4,\ldots,s\}$, and 
at most $(k-1)+(k-1)-1$ vertices in $(N(u_1)\cup N(u_2))\setminus \{v_1,v_2\}$ since 
$u_1z,u_2w,zw\in E$. Note that $I$ is an independent dominating set. We derive
$$i(G)\leq  |I| \leq  \frac{\g(G)}{2}+\left(\frac{\g(G)}{2}-2 \right)(k-1)+2k-3<\frac{\g(G)}{2}k,$$
leading to a contradiction again.

\medskip

With this we have shown that no neighbor of $u_1$ in $B$ has a neighbor outside $N(u_1)\cup N(v_1)$.
Proceeding analogously for neighbors of $v_1$ we see that no vertex of $N(u_1)\cup N(v_1)$ has a neighbor outside $N(u_1)\cup N(v_1)$.
That is, $G[N(u_1)\cup N(v_1)]$ is a component of $G$,
and since $G$ is connected, we have $\gamma(G)=2$.

Now suppose that there are vertices $z_1,z_2\in N(u_1)\setminus\{v_1\}$ such that $z_1z_2\in E$.
Then put to an independent set $v_1$ and complete it to an independent dominating set of $G$ with neighbors of $u_1$ in $V\setminus A$.
Since $z_1z_2\in E$, we get an independent dominating set of size at most $1+(k{-}2)<k=\gamma(G)\frac k2$, which gives $\frac{i(G)}{\gamma(G)}<\frac k2$ again.
Thus, $N(u_1)\setminus \{v_1\}$ is an independent set in $G$.
In fact, $N(u_1)$ itself is an independent set of $G$, since we have already shown that $v_1$ has no neighbors in $N(u_1)\cap V\setminus A$.
Analogously it can be shown that $N(v_1)$ is an independent set, which means that $G$ is a $k$-regular graph with $2k$ vertices, and two independent sets $N(u_1)$ and $N(v_1)$ of size $k$.
Consequently, $G$ is $K_{k,k}$.
\end{proof}

\vskip 1pc \noindent{\bf Acknowledgements.} The first author
acknowledges partial support by Slovak research grants VEGA
1/0142/17, VEGA 1/0238/19, APVV-15-0220 and APVV-17-0428.
The research was partially supported also by Slovenian research agency
ARRS, programs no.\ P1--0383 and J1--1692.


\begin{thebibliography}{99}

\bibitem{babi}
A. Babikir, M. A. Henning, 
Domination versus independent domination in graphs of small regularity,
{\it Discrete Math.\/} \textbf{343} (2020) 111727.

\bibitem{bare} 
C. Barefoot, F. Harary, K. F. Jones,
What is the difference between the domination and independent
domination numbers of a cubic graph? 
{\it Graphs Combin.\/} \textbf{7} (1991) 205--208. 

\bibitem{cock-hed} 
E. J. Cockayne, S. T. Hedetniemi, 
Independence and domination in 3-connected cubic graphs, 
{\it J. Combin. Math. Combin. Comput.\/} \textbf{10} (1991) 173--182. 

\bibitem{furuya}
M. Furuya, K. Ozeki, A. Sasaki,
On the ratio of the domination number and the independent
domination number in graphs,
{\it Discrete Applied Math.\/} \textbf{178} (2014) 157--159.


\bibitem{GHsur}
W. Goddard, M. A. Henning,
Independent domination in graphs: A survey and recent results, 
{\it Discrete Math.\/} \textbf{313} (2013) 839--854.

\bibitem{GHLS} W. Goddard, M. A. Henning, J. Lyle, J. Southey, On the independent domination number of regular graphs, {\it Ann. Combin.\/} \textbf{16} (2012) 719--732.

\bibitem{Hay1}
T. W. Haynes, S. T. Hedetniemi, P. J. Slater, 
\textit{Domination in Graphs:
Advanced Topics}, Marcel Dekker Inc., New York, 1998.

\bibitem{Hay2}
T. W. Haynes, S. T. Hedetniemi,  P. J. Slater, 
\textit{Fundamentals of Dom-
ination in Graphs}, Marcel Dekker Inc., New York, 1998.

\bibitem{kost} A.V. Kostochka,
The independent domination number of a cubic 3-connected graph can be much larger than its domination number, 
{\it Graphs Combin.\/} \textbf{9} (1993) 235--237.

\bibitem{O}
S. O, D. B. West, 
Cubic graphs with large ratio of independent domination number to domination number,
{\it Graphs Combin.\/} \textbf{3} (2016) 773--776.

\bibitem{rose64} 
M. Rosenfeld, 
Independent sets in regular graphs, 
{\it Israel J. Math.\/} \textbf{2} (1964) 262--272.

\bibitem{SH} 
J. Southey, M. A. Henning, 
Domination versus independent domination in cubic graphs,
{\it Discrete Math.\/} \textbf{313} (2013) 1212--1220.



\end{thebibliography}
\end{document}